\documentclass[12pt]{article}
\pdfoutput=1

\usepackage[utf8]{inputenc}
\usepackage{geometry}
\usepackage{amsmath,amsfonts,amsthm,amscd,upref,amstext, amssymb}
\usepackage{verbatim}
\usepackage[shortlabels]{enumitem}
\usepackage{xcolor}

\usepackage[
backend=biber,
style=alphabetic,
]{biblatex}
\title{A bibLaTeX example}
\usepackage{url}
\newtheorem{theorem}{Theorem}[section]

\newtheorem{lemma}[theorem]{Lemma}
\newtheorem{claim}[theorem]{Claim}
\newtheorem{subclaim}[theorem]{Subclaim}

\newtheorem{definition}[theorem]{Definition}
\newtheorem{conjecture}[theorem]{Conjecture}
\newtheorem{remark}[theorem]{Remark}

\newcommand\T{\mathcal{T}}
\newcommand\R{\mathbb{R}}

\newcommand\U{\mathcal{U}}

\renewcommand\P{\mathbb{P}}

\usepackage{fancyhdr}
\usepackage{lipsum} 

\fancypagestyle{plain}{
  \fancyhf{}
  \fancyfoot[C]{\thepage}%
}
\fancypagestyle{firstpage}{
  \fancyhf{}
  \fancyfoot[L]{}%
}
\title{Derived Models in PFA}
\author{Derek Levinson and Nam Trang}
\date{\today}

\begin{document}
\maketitle
\thispagestyle{firstpage}

\begin{abstract}
    We discuss a conjecture of Wilson that under the Proper Forcing Axiom, $\Theta_0$ of the derived model at $\kappa$ is below $\kappa^+$. We prove the conjecture holds for the old derived model. Assuming mouse capturing in the new derived model, the conjecture holds there as well. We also show $\Theta < \kappa^+$ in the case of the old derived model, and under additional hypotheses for the new derived model.
\end{abstract}

\section{Introduction}

In this paper, we prove several results which represent progress on a conjecture of Trevor Wilson relating large cardinals, derived models, and the Proper Forcing Axiom (PFA).

Derived models were invented as a source of models of the axiom of determinacy (AD). Suppose $\kappa$ is a limit of Woodin cardinals and $G$ is $Col(\omega,<\kappa)$-generic. A relatively simple model of $AD$ is $L(\R^*_G)$, where $\R^*_G$ represents the reals in the symmetric collapse of $V$ induced by $G$. There may, however, be larger models of $AD$ in $V[G]$. The ``old derived model'' at $\kappa$ includes $\R^*_G$ as well as a large collection of set of reals $Hom^*_G$ not contained in $L(\R^*_G)$ and, like $L(\R^*_G)$, can be shown to satisfy $AD^+$. The ``new derived model'' includes even more set of reals and is roughly the largest model of $AD^+$ contained in $V(\R^*_G)$.

Derived models have been extensively studied in the case that $V$ is a mouse (see \cite{dmatm}). Much less is known about the derived model if $V$ lacks the fine-structural properties of mice. In the opposite direction, we are interested in the derived model in the case where $V$ is very ``wide,'' for example when forcing axioms hold in $V$. Recent progress was made here by Wilson:

\begin{theorem}[Wilson]
\label{wilson thm}
    (PFA) If $\kappa$ is a limit of Woodin cardinals of countable cofinality, then $\Theta_0$ of the derived model at $\kappa$ is below $\kappa^+$.
\end{theorem}

In other words, the theorem says functions from $\R^*_G$ into $\kappa^+$ which are ordinal definable in the derived model are bounded in $\kappa^+$. On the other hand, $V[G] \models |\R^*_G| = \kappa$. So the theorem implies the ordinal definable part of the derived model is not close to being all of $V[G]$.

In a sequel to this paper we will show the hypotheses of Theorem \ref{wilson thm} imply $\Theta$ of the derived model is below $\kappa^+$. For now, we are interested in a different generalization of this theorem. Wilson conjectured the assumption of countable cofinality in Theorem \ref{wilson thm} is unnecessary:

\begin{conjecture}[Wilson]
\label{first part of conjecture}
    (PFA) If $\kappa$ is a limit of Woodin cardinals, then $\Theta_0$ of the derived model at $\kappa$ is below $\kappa^+$.
\end{conjecture}

We make substantial progress on Conjecture \ref{first part of conjecture}. First, we prove Conjecture \ref{first part of conjecture} for the old derived model. In fact, we show something stronger:

\begin{theorem}
\label{intro thm for old}
    $(PFA)$ If $\kappa$ is a limit of Woodin cardinals, then $\Theta$ of the old derived model at $\kappa$ is below $\kappa^+$.
\end{theorem}

We also prove the conjecture for the new derived model assuming mouse capturing:

\begin{theorem}
\label{intro thm for new}
    $(PFA)$ If $\kappa$ is a limit of Woodin cardinals and the new derived model satisfies mouse capturing, then $\Theta_0$ of the new derived model at $\kappa$ is below $\kappa^+$.
\end{theorem}

Under additional assumptions, we also show $\Theta_\alpha$ of the new derived model is below $\kappa^+$ for larger $\alpha$.

\section{Background}
\subsection{The Derived Model}

\begin{definition}
    Suppose $\kappa$ is a limit of Woodin cardinals and $G$ is $Col(\omega,<\kappa)$-generic over $V$. Let $\R^*_G = \bigcup_{\gamma < \kappa} \R^{V[G\upharpoonright\gamma]}$ and $V(\R^*_G) = HOD^{V[G]}_{V\cup \R^*_G}$. Then
    \begin{enumerate}
        \item The ``old'' derived model is $olD(V,\kappa) = L(Hom^*_G,\R^*_G)$, where 
        \begin{align*}
            Hom^*_G = \{ & \rho[T] \cap \R^*_G : (\exists \gamma < \kappa) \, T\in V[G\upharpoonright\gamma] \text{ and } \\
            & V[G\upharpoonright\gamma] \models ``T \text{ is } <\kappa-\text{absolutely complemented''}\}.
        \end{align*}
        
        \item The ``new'' derived model is $D(V,\kappa) = L(\mathcal{A},\R^*_G)$, where 
        \begin{align*}
            \mathcal{A} = \{A\subseteq \R^*_G: A\in V(\R^*_G) \wedge L(A,\R^*_G) \models AD^+\}.
        \end{align*}

    \end{enumerate}
\end{definition}

Both $olD(V,\kappa)$ and $D(V,\kappa)$ satisfy $AD^+$.\footnote{See \cite{dmt} for a proof in the case of the old derived model.} $olD(V,\kappa)$ is our own notation to circumvent the unfortunate convention that the old and new derived models are typically not distinguished. Both $olD(V,\kappa)$ and $D(V,\kappa)$ technically depend on the generic $G$, but as their theories are independent of the generic, the notation makes no reference to $G$. We will also omit $G$ as a subscript to $\R^*$, $Hom^*$, and $V(\R^*)$ where it is convenient to do so.

\begin{remark}
\label{same suslin-co-suslin sets}
    The old derived model is contained in the new derived model. On the other hand, the Suslin-co-Suslin sets of the two are the same. For suppose $A\subseteq\R^*$ is Suslin-co-Suslin in $D(V,\kappa)$, so that $A = \rho[T] = \rho[S]^c$ for some $T,S\in V(\R^*)$. Then $T,S$ are $OD$ in $V[G]$ from $s$ for some finite $s\subset ON \cup \R^*$. Pick $\gamma$ large enough that $s\in V[G\upharpoonright\gamma]$. Then $T,S\in V[G\upharpoonright\gamma]$ and are $<\kappa$-absolutely complementing.
\end{remark}

\begin{definition}
    For $A\subseteq \R$, $\Theta(A) = sup\{\alpha\in On : \text{ exists surjection } f:\R\to \alpha \text{ which is } OD(A,x) \text{ for some } x\in\R\}$. We set
    \begin{enumerate}
        \item $\Theta_0 = \Theta(\emptyset)$,
        \item $\Theta_{\alpha+1} = \Theta(A)$ for any $A$ such that $w(A) = \theta_\alpha$, and
        \item $\Theta_\lambda = sup_{\alpha < \lambda} \Theta_\alpha$ for $\lambda$ a limit ordinal.
    \end{enumerate}
\end{definition}

Note $\Theta_{\alpha+1}$ is not defined if $\Theta_\alpha = \Theta$. The collection $\langle \Theta_\alpha : \alpha \leq \beta\rangle$ where $\beta$ is least so that $\Theta_\beta = \Theta$ is called the Solovay hierarchy.

\begin{remark}
\label{AD_R equivalences}
    Under $AD^+$ the following are equivalent:
    \begin{enumerate}
        \item $AD_\R$
        \item Every set of reals is Suslin-co-Suslin.
        \item $\Theta$ is a limit level of the Solovay hierarchy.
    \end{enumerate}
\end{remark}

By Remarks \ref{same suslin-co-suslin sets} and \ref{AD_R equivalences}, if $D(V,\kappa) \models AD_\R$, then $D(V,\kappa) = olD(V,\kappa)$.

\subsection{Wilson's Conjecture}

\begin{conjecture}[Wilson]
\label{trevor conjecture}
Assume $PFA + ``\kappa$ is a limit of Woodin cardinals.'' Then
\begin{enumerate}
    \item \label{1st conjecture} $\Theta_0^{D(V,\kappa)} < \kappa^+$ and
    \item \label{2nd conjecture} $\Theta_0^{D(V,\kappa)} < \Theta^{D(V,\kappa)}$.
\end{enumerate}
\end{conjecture}

\begin{remark}
    The \ref{2nd conjecture}nd part of Wilson's conjecture implies the \ref{1st conjecture}st.
\end{remark}

\begin{remark}
    Assume $PFA + ``\kappa$ is a limit of Woodin cardinals'' is consistent. The assumption of $PFA$ in Conjecture \ref{trevor conjecture} in necessary.
\end{remark}
\begin{proof}
    Let $M$ be a model of $ZF + AD$. We may assume $M\models V = L(\R)$, so that $\Theta_0^M = \Theta^M$. Let $N$ be a Prikry-generic premouse over $M$ and let $\delta_\infty$ be the supremum of the Woodin cardinals of $N$.\footnote{See Chapter 6.6 of \cite{hacm}.} We will show $\Theta^{D(N,\delta_\infty)} = (\delta_\infty^+)^N$.

    There exists $H$ a $Col(\omega,<\delta_\infty)$-generic over $N$ such that $\R^*_H = \R^M$.\footnote{See Claim 6.46 of \cite{hacm}.} In particular, $L(\R^*_H) = M$. We must show $(\delta_\infty^+)^N = \Theta_0^M$. Clearly, $(\delta_\infty^+)^N \geq \Theta^M$.
    
    $N\subset M[H]$, where $H$ is $\P$-generic for $\P$ the Prikry-forcing with the Martin measure. Since $\P$ is $\Theta^M$-c.c, $\Theta^M$ is a cardinal of $N$. In particular, $(\delta_\infty^+)^N \leq \Theta^M$.
\end{proof}

Wilson proved Theorem \ref{wilson thm} by constructing a coherent covering matrix for $\kappa^+$ in $V$ in the event that $\Theta_0$ of the derived model is $\kappa^+$. But by a theorem of Viale, there is no coherent covering matrix for $\kappa^+$ assuming $PFA$. For $\kappa$ of uncountable cofinality, we cannot use coherent covering matrices. Instead, we'll make use of the failure of square principles. 

\subsection{Square Sequences}

\begin{definition}
    We call $\langle C_\alpha : \alpha < \lambda\rangle$ a coherent sequence iff for all limit $\alpha< \lambda$, $C_\alpha$ is a club subset of $\alpha$ and $C_\beta = \beta \cap C_\alpha$ whenever $\beta \in lim(C_\alpha)$.
\end{definition}

\begin{definition}
    We say $\square_\kappa$ holds if there is a coherent sequence $\langle C_\alpha : \alpha < \kappa^+\rangle$ such that $cof(\alpha) < \kappa \implies |C_\alpha| < \kappa$.
\end{definition}

\begin{definition}
    Suppose $S\subseteq \lambda$ is a club. We say a sequence $\vec{D} = \langle D_\alpha : \alpha \in S\rangle$ is almost coherent if for all $\alpha,\beta\in S$
    \begin{enumerate}
        \item $D_\alpha \subseteq S$ and is a closed subset of limit ordinals below $\alpha$, 
        \item $cof(\alpha)>\omega \implies D_\alpha$ is unbounded in $\alpha$, and
        \item $\beta \in D_\alpha \implies D_\beta = \beta \cap D_\alpha$.
    \end{enumerate}
\end{definition}

\begin{definition}
\label{square prime definition}
    We say $\square'_\kappa$ holds if there is a sequence $\vec{C} = \langle C_\alpha : \alpha < \kappa^+\rangle$ such that $\vec{C}$ is almost coherent and for all $\alpha < \kappa^+$, $ot(C_\alpha) \leq \kappa$.
\end{definition}

\begin{remark}
\label{square' on club implies square'}
    Suppose $ZF + \square'_\kappa$ holds on a club --- that is, there is a club $S\subseteq \kappa^+$ and an almost coherent sequence $\vec{C} = \langle C_\alpha : \alpha\in S\rangle$ so that $\alpha\in S \implies ot(C_\alpha) \leq\kappa$. Then $\square'_\kappa$ holds.
\end{remark}

\begin{remark}
\label{square' implies square}
    $ZFC + \square'_\kappa$ implies $\square_\kappa$.\footnote{See Lemma 5.1 of Chapter III of \cite{Devlin}.}
\end{remark}

For our main theorems, we will build $\square_\kappa$-sequences in the stages suggested by the definitions above. Inside a derived model, we will construct an almost coherent sequence. The particular way our almost coherent sequence is constructed will allow us to turn it into a $\square'_\kappa$-sequence on a club by a technique from \cite{zssquare} and Remark \ref{square' on club implies square'}. Then in a model of $ZFC$ we can rearrange this as a $\square_\kappa$-sequence by Remark \ref{square' implies square}.

\section{Result in Old Derived Model}

\begin{theorem}
\label{result in old derived model}
    Suppose $\kappa$ is a limit of Woodin cardinals and $\neg \square_\kappa$. Then $\Theta^{olD(V,\kappa)} < \kappa^+$.
\end{theorem}
\begin{proof}
    Let $G$ be $Col(\omega,<\kappa)$-generic over $V$..

    \begin{claim}
    \label{size of hom}
        There is a $Col(\omega,<\kappa)$-name $\pi$ for $(Hom^*,\R^*)$ and a code $CODE$ for $\pi$ such that $CODE \subset H^V_\kappa$.
    \end{claim}
    \begin{proof}
        If $x\in \R^*$ then $x\in V[G\upharpoonright \alpha]$ for some $\alpha < \kappa$, so $x$ has a $Col(\omega,<\kappa)$-name in $H^V_\kappa$. Then there is a name $\pi_2$ for $\R^*$ contained in $H^V_\kappa$. And any $A^* \subseteq \R^*$ has a $Col(\omega,<\kappa)$-name $\dot{A}^*$ contained in $H^V_\kappa$ (if $\sigma$ is any name for $A^*$, we can take $\dot{A}^* = \{(\dot{x},p): \dot{x}\in H^V_\kappa \wedge p \Vdash^V_{Col(\omega,<\kappa)} \dot{x}\in\sigma\}$).

        Since $\kappa$ is a strong limit cardinal in $V[G\upharpoonright\alpha]$, $|Hom_{<\kappa}^{V[G\upharpoonright\alpha]}|^{V[G\upharpoonright\alpha]} < \kappa$. Then let $\mathcal{A_\alpha}\in V$ be a $Col(\omega,\gamma)$-name for $Hom_{<\kappa}^{V[G\upharpoonright\alpha]}$ such that $|\mathcal{A}_\alpha| <  \kappa$. For each $A\in Hom_{<\kappa}^{V[G\upharpoonright\alpha]}$, there is a unique $A^*\subseteq V[G]$ such that $A^* = \rho[T]$ for some tree $T$ witnessing $A\in Hom_{<\kappa}^{V[G\upharpoonright\alpha]}$. Let $\mathcal{A}_\alpha^*$ be a $Col(\omega,<\kappa)$-name (in $V$) such that
        \begin{enumerate}
            \item $|\mathcal{A}_\alpha^*| < \kappa$,
            \item $\bigcup \mathcal{A}_\alpha^*\subset H^V_\kappa$, and
            \item $\emptyset \Vdash (\mathcal{A}_\alpha^*)_G = \{(A_{G\upharpoonright\alpha})^* : A \in \mathcal{A}_\alpha\}$.
        \end{enumerate}

        Let $\pi_2 = \bigcup_{\alpha < \kappa} \mathcal{A}^*_\alpha$.

        Let $\pi$ be a natural\footnote{I.e. defined in some reasonable way.} name for the pair $(\pi_1,\pi_2)$. $\pi$ is a name for $(Hom^*,\R^*)$ and since $|trcl(\pi)| = \kappa$, $\pi$ can be coded by a set $CODE \subset H^V_\kappa$.
    \end{proof}

    Let $M = L(H^V_\kappa,CODE)$. Then $olD(V,\kappa)\subseteq M[G]$, since $\pi_1,\pi_2\in M[G]$, $\pi_1[G] = \R^*$, and $\pi_2[G] = Hom^*$.

    Suppose $\Theta^{olD(V,\kappa)} = \kappa^+$. Then $\Theta^{M[G]} = \kappa^+$, since $olD(V,\kappa)\subseteq M[G] \subseteq V[G]$. Since $G$ does not collapse any cardinals above $\kappa$, it follows that $\Theta^M = \kappa^+$. Then in $M$, we have a club $S\subseteq \kappa^+$ and an almost coherent sequence $\vec{D} = \langle D_\alpha : \alpha \in S\rangle$. This is by the proof of \cite{coherent}: Replace $Lp^{^G\Sigma}(\R)$ by $L(H^V_\kappa \cup CODE)$ in their argument. Then $S$ and $D_\alpha$ are defined just as in \cite{coherent}, except we restrict $S$ to ordinals greater than $\kappa$ and only include ordinals above $\kappa$ in $D_\alpha$.
    
    Note $M\subseteq V$, so $\vec{D}\in V$. Working in $V$, we will transform $\vec{D}$ into a sequence $\vec{C}$ of length $\kappa^+$ realizing $\square_\kappa$ holds. This is by now a standard argument, but we outline it below for the reader's convenience. The first step is to get a $\square'_\kappa$-sequence on a club.
    
    \begin{lemma}
    \label{square' from old derived model}
        There is a club $S\subset \kappa^+$ and a sequence $\vec{C}' = \langle C'_\alpha : \alpha \in S \rangle$ such that
        \begin{enumerate}
            \item $C'_\alpha$ is a closed set of ordinals below $\alpha$, 
            \item $cof(\alpha)>\omega \implies C'_\alpha$ is unbounded in $\alpha$,
            \item $ot(C'_\alpha)\leq \kappa$, and
            \item $\beta \in C'_\alpha \implies C'_\beta = \beta \cap C'_\alpha$.
        \end{enumerate}
    \end{lemma}
    \begin{proof}
        Recall our sequence $\vec{D}$ was obtained from \cite{coherent}. We adopt the notation of that proof. In particular, for $\tau\in S$,

        \begin{itemize}
            \item $N_\tau$ is the least initial segment of $L(H_\kappa^V \cup CODE)$ such that there is $n_\tau<\omega$ satisfying $\rho^{N_\tau}_{n_\tau} \geq \tau$ and $\rho^{N_\tau}_{n_\tau+1} = H_\kappa^V \cup CODE$,
            \item $p_\tau = p^{N_\tau}_{n_\tau}$,
            \item $\tilde{h}_\tau$ is the $\Sigma_1^{(n_\tau)}$ Skolem function for $N_\tau$,
            \item and if $\bar{\tau}\in D_\tau$, then $\sigma_{\bar{\tau}\tau}: N_{\bar{\tau}}\to N_\tau$ is a $\Sigma_0^{(n_\tau)}$-preserving map such that $crit(\sigma_{\bar{\tau}\tau}) = \bar{\tau}$, $\sigma_{\bar{\tau}\tau}(\bar{\tau}) = \tau$, and $\sigma_{\bar{\tau},\tau}(p_{\bar{\tau}}) = p_\tau$.
        \end{itemize}

        We will follow closely the proof of Lemma 3.6 of \cite{zssquare}.

        Let $\langle x_\alpha : \alpha < \kappa\rangle$ be an enumeration of $H_\kappa^V \cup CODE$.\footnote{This exists (in $V$) because $\kappa$ is a limit of Woodin cardinals and $CODE\subset H_\kappa^V$. We need to work in $V$ for this so that we can use $AC$, which is why we defined our sequence $\vec{D}$ from $L(H^V_\kappa \cup CODE)$ instead of from $L(Hom^*,\R^*)$.} Let $X_\tau(\xi)$ be the $\Sigma^{(n_\tau)}_1$-hull of $\{x_\xi,p_\tau\}$ in $N_\tau$. Define the sequence $\langle \tau_\iota, \xi_\iota\rangle$ as follows:
        \begin{enumerate}
            \item $\tau_0 = \min(D_\tau\cup\{\tau\})$.
            \item $\xi^\tau_\iota = \text{least } \xi<\kappa$ s.t. $X_\tau(\xi)\nsubseteq \text{range}(\sigma_{\tau_\iota\tau})$.
            \item $\tau_{\iota+1} = \text{least } \bar{\tau}\in D_\tau\cup\{\tau\}$ s.t. $X_\tau(\xi^\tau_\iota)\subseteq \text{range}(\sigma_{\bar{\tau}\tau})$.
            \item For limit $\gamma$, $\tau_\gamma = \sup\{\tau_\iota:\iota < \gamma\}$.
            \item $\iota_\tau = \text{least } \iota$ s.t. $\tau_\iota = \tau$.
        \end{enumerate}

        \begin{remark}
            $\tau_\iota < \tau \implies \xi^\tau_\iota$ exists. In particular, the construction does not stop until it reaches $\iota_\tau$.
        \end{remark}
        \begin{proof}
            Let $\xi < \kappa$ be s.t. $\tau_\iota = \tilde{h}_\tau(x_\xi,p_\tau)$. Then $X_\tau(\xi) \nsubseteq \text{range}(\sigma_{\tau_\iota \tau})$ since $\tau_\iota = crit(\sigma_{\tau_\iota \tau})$.
        \end{proof}

        Let $C'_\tau = \{\tau_\iota :\iota < \iota_\tau\}$. One can show $\vec{C}' = \langle C'_\tau : \tau \in S\rangle$ realizes the lemma just as in \cite{zssquare}.

    \end{proof}
    
    By Lemma \ref{square' from old derived model}, together with Remarks \ref{square' on club implies square'} and \ref{square' implies square}, $V\models \square_\kappa$, a contradiction.
\end{proof}

Theorem \ref{intro thm for old} is immediate from Theorem \ref{result in old derived model} and that $PFA \implies \neg \square_\kappa$.

\section{Results in New Derived Model}
\label{new derived model section}

Our results in the new derived model rely upon mouse capturing and other similar principles.

\begin{definition}
    Assume $AD^+$. Mouse capturing (MC) holds if for every $x\in \R$ and every $y\in OD(x)\cap \R$, there is an $x$-mouse $M$ such that $y\in M$.
\end{definition}

The mouse set conjecture says $MC$ holds in every model of $AD^+ + V = L(P(\R))$. In particular, we expect $MC$ is true in any derived model. We use only the following consequence of mouse capturing. Let 
\begin{align*}
    Lp(A) = \bigcup \{M : M \text{ is a sound } A \text{-mouse projecting to } A\}.
\end{align*}
$Lp(A)$ can be reorganized as an $A$-premouse. Assuming $AD^+ + V = L(P_{\Theta_0}(\R)) + MC$, $V = L(Lp(\R))$.\footnote{This is a special case of the result of \cite{mscfsor}.}

Our first result of this section gives Theorem \ref{intro thm for new}.

\begin{theorem}
\label{theta_0 < kappa+}
    Suppose $V\models \neg \square_\kappa$, $\kappa$ is a limit of Woodin cardinals and $D(V,\kappa)\models MC$. Then $\Theta_0^{D(V,\kappa)} < \kappa^+$.
\end{theorem}
\begin{proof}

Let $M = D(V,\kappa)$ and suppose for contradiction $\Theta_0^M = \kappa^+$. Let $G$ be the $Col(\omega,<\kappa)$-generic used in constructing $M$. Of course $\Theta^M \leq (\kappa^+)^{V[G]} \leq (\kappa^+)^V$, so our assumption implies $\Theta^M = \Theta_0^M$. Any derived model satisfies $V = L(P(\R^*))$, so in this case $M \models V = L(P_{\Theta_0}(\R^*))$. Then by the assumption of mouse capturing, $M = L(Lp(\R^*))$. In particular, the height of $Lp(\R^*)^M$ is $(\kappa^+)^V$.

Let $\pi$ be a $Col(\omega,<\kappa)$-name for $\R^*$ (in $V$) such that $\pi\subset H_\kappa^V$. Let $S$ be the output of the $S$-construction\footnote{First defined under the name $P$-construction in \cite{self-iter}.} in $Lp(\R^*)$ over $\pi\cup H^V_\kappa$. Then $S\in V$ and $S\trianglelefteq Lp(\pi\cup H^V_\kappa)$. In particular, $Lp(\pi \cup H^V_\kappa)$ has height $\kappa^+$.

By the argument of \cite{coherent}, there is (in $V$) a club $C\subseteq \kappa^+$ and an almost coherent sequence $\vec{D} = \langle D_\alpha : \alpha \in C\rangle$.\footnote{ We can apply the argument of \cite{coherent} because none of the extenders on the extender sequence of $S$ have critical point less than $\kappa$.}

$\vec{D}$ can be transformed into a $\square_\kappa$-sequence as in the proof of Theorem \ref{result in old derived model}, contradicting that $V\models \neg \square_\kappa$.

\end{proof}

We next adapt the technique above to show $\Theta_{\alpha+1} < \kappa^+$ if in place of our mouse capturing assumption from the previous theorem, we assume $D(V,\kappa) = L(Lp^\Sigma(\R))$ for a sufficiently nice iteration strategy $\Sigma$.

\begin{theorem}
\label{theta < kappa+}
    Suppose $V\models \neg \square_\kappa$, $\kappa$ is a limit of Woodin cardinals and $D(V,\kappa)\models ``\Theta_{\alpha+1}$ exists'' $ + $``there is a hod pair $(P,\Sigma)$ such that $P_{\Theta_{\alpha+1}}(\R^*) = Lp^{\Sigma}(\R^*) \cap P(\R^*)$, $\Sigma$ is fullness-preserving, and has branch condensation.\footnote{In the sense of Definitions 1.36, 0.20, and 2.14 of \cite{hm&msc}.} Then $\Theta_{\alpha+1}^{D(V,\kappa)} < \kappa^+$.
\end{theorem}

The assumptions on $D(V,\kappa)$ are known to hold in any model of $AD^+$ in which $\Theta_{\alpha+1}$ exists and a sufficient smallness condition on $D(V,\kappa)$ holds (see e.g. \cite{hm&msc}).

\begin{proof}[Proof of Theorem \ref{theta < kappa+}]
    Let $M = D(V,\kappa)$ and suppose $\Theta_{\alpha+1}^M = \kappa^+$. Then $M\models V = L(P_{\Theta_{\alpha+1}}(\R^*))$. Let $(P,\Sigma)$ be as in the statement of the theorem. So $M = L(Lp^{\Sigma}(\R^*))$.

    Let $G$ be the $Col(\omega,<\kappa)$-generic used to construct $M$, so that $\R^* = \R^M = \bigcup_{\gamma<\kappa}\R^{V[G\upharpoonright\gamma]}$.

    Note $P$ is coded by a real and $\Sigma$ is Suslin-co-Suslin in $M$.\footnote{See Theorem 5.9 of \cite{hm&msc}.} Then there is $\gamma < \kappa$ and a tree $T\in V[G\upharpoonright\gamma]$ such that $(\rho[T])^M$ codes $(P,\Sigma)$.\footnote{$T$ can be taken to be $<\kappa$-absolutely complemented, but we don't need this property.}

    \begin{claim}
        $\Sigma\upharpoonright V[G\upharpoonright\gamma]\in V[G\upharpoonright\gamma]$
    \end{claim}
    \begin{proof}
        
        Let $S$ be an iteration tree (according to $\Sigma$) of limit length on $P$, with $S\in V[G\upharpoonright\gamma]$. Let $b$ be the branch through $S$ in $V[G]$ selected by $\rho[T]$. Suppose $H$ is another $Col(\omega,<\kappa)$-generic such that $H\upharpoonright\gamma = G\upharpoonright\gamma$. Then there is a branch $b'\in V[H]$ through $S$ chosen by $\rho[T]$. But we can build a third $Col(\omega,<\kappa)$ generic $H'$ such that $H'\upharpoonright\gamma = G\upharpoonright\gamma$ and  $G$, $H \in V[H']$. Then $b$ and $b'$ are both branches through $S$ according to $\rho[T]$, so $b = b'$. It follows that $b\in V[G\upharpoonright\gamma]$.
    \end{proof}

    \begin{remark}
        If $(P,\Sigma) = \rho[T]$ for some $T\in V$, then we could proceed as in the proof of Theorem \ref{theta_0 < kappa+} to obtain a $\square_\kappa$-sequence in $V$. Instead, we first have to modify $(P,\Sigma)$ to get a pair $(Q,\Lambda)$ which is sufficiently definable in $V$, so that the coherent sequence induced by $Lp^{\Lambda}$ is definable in $V$. To do this, we'll use a boolean-valued comparison argument similar to the one in \cite{cmi}.
    \end{remark}

    \begin{claim}
    \label{getting sigma from lambda}
        Suppose $U$ is a non-dropping iteration tree on $P$ according to $\Sigma$ with last model $Q$ and $\Lambda = \Sigma_{Q,U}$.\footnote{I.e. $\Lambda$ is the tail strategy.} Then $Lp^{\Sigma}(\R^*) \subseteq Lp^{\Lambda}(\R^*)$.
    \end{claim}
    \begin{proof}
        \begin{subclaim}
        \label{strategy for M1Sigma}
            The iteration strategy for $M_1^{\Sigma,\#}$ (restricted to trees in $Lp^{\Lambda}(\R^*)$) is definable in $Lp^{\Lambda}(\R^*)$.
        \end{subclaim}
        \begin{proof}
            Suppose $T$ is a tree on $P$ according to $\Sigma$ of limit length and $T\in Lp^{\Lambda}(\R^*)$. Lift $T$ to a tree $T'$ on $Q$ (this can be done in $Lp^{\Lambda}(\R^*)$ since the iteration from $P$ to $Q$ is countable in $V[G]$ and thus coded by a real). Since $\Lambda$ is a tail strategy of $\Sigma$ and $\Sigma$ has branch condensation, $\Sigma(T)$ is definable from $\Lambda(T')$. Since $\Lambda(T')$ is definable in $Lp^{\Lambda}(\R^*)$ from $T$, $\Sigma(T)$ is as well.

            It follows that the operation $a \mapsto M_0^{\Sigma,\#}(a)$ (for transitive $a\in Lp^{\Lambda}(\R^*)$) is definable in $Lp^{\Lambda}(\R^*)$. $M_0^{\Sigma,\#}(a)$ is the minimal active, sound, $\Sigma$-premouse over $a$ projecting to $a$ such that any iterate of $M_0^{\Sigma,\#}$ by hitting its top extender is a $\Sigma$-premouse. This is definable in $Lp^{\Lambda}(\R^*)$ from $a$ since $\Sigma\upharpoonright Lp^{\Lambda}(\R^*)$ is definable in $Lp^{\Lambda}(\R^*)$.

            Then we have an iteration strategy for $M_1^{\Sigma,\#}$ in $Lp^{\Lambda}(\R^*)$ using $Q$-structures.
        \end{proof}

        The claim is immediate from Subclaim \ref{strategy for M1Sigma} and inspecting the definition of $Lp^{\Lambda}$.
    \end{proof}

    Let $G = H \times L$, where $H$ is $Col(\omega,\gamma)$-generic and $L$ is $Col(\omega,<\nu)$-generic for some $\nu$ such that $\gamma + \nu = \kappa$. Let $\tau$ be a $Col(\omega,\gamma)$-name for $T$ such that
    \begin{equation*}
        \emptyset \Vdash^V_{Col(\omega,\gamma)}``\Vdash^{V[\dot{G}\upharpoonright\gamma]}_{Col(\omega,<\kappa)} \rho[\tau] \text{ satisfies all properties of } (P,\Sigma) \text{ mentioned above.''}
    \end{equation*}

    Let $\tau_1$ and $\tau_2$ be names such that $\emptyset \Vdash ``\rho[\tau] \cap \R^*$ codes $(\tau_1,\tau_2)$.'' So $\tau_1$ is a name for a hod mouse with strategy $\tau_2$. For $q\in Col(\omega,\gamma)$, let $H_q = q \cup H\upharpoonright \omega\backslash domain(q)$. Note $V[H_q] = V[H]$ for any $q\in Col(\omega,\gamma)$. Let $P_q = \tau_1^{H_q}$ and $\Sigma_q = \tau_2^{H_q}$. Note $P_q \in V[H]$ and while $\Sigma_q\in M$, $\Sigma_q\upharpoonright V[H] \in V[H]$.

    \begin{claim}
    \label{can coiterate}
        If $q_1,q_2\in Col(\omega,\gamma)$, then there are countable stacks $\U_1$ and $\U_2$ in $V[H]$ such that (for $i\in \{1,2\}$) $U_i$ is a stack on $P_{q_i}$ according to $\Sigma_{q_i}$ and letting $Q_i$ be the last model of $U_i$ and $\Lambda_i = (\Sigma_{q_i})_{Q_i,U_i}$, $Q_1 = Q_2$ and $\Lambda_1 = \Lambda_2$.
    \end{claim}
    \begin{proof}
        Since $\Sigma_{q_1}$ and $\Sigma_{q_2}$ are both $\Gamma$-fullness preserving for $\Gamma = D(V,\kappa) \cap P(\R)$ and have branch condensation, $(P_{q_1},\Sigma_{q_1})$ and $(P_{q_2},\Sigma_{q_2})$ are of the same kind.\footnote{Lemma 3.32 of \cite{hm&msc} proves this in the (harder) case that the order type of the Woodin cardinals in $P_{q_1}$ and $P_{q_2}$ are limit ordinals.} Then $V[H]$ satisfies $(P_{q_1},\Sigma_{q_1}\upharpoonright V[H])$ and $(P_2,\Sigma_{q_2}\upharpoonright V[H])$ are of the same kind, so we may apply Theorem 2.47 of \cite{hm&msc} in $V[H]$ to $(P_{q_1},\Sigma_{q_1}\upharpoonright V[H])$ and $(P_2,\Sigma_{q_2}\upharpoonright V[H])$. Let $U_1$ and $U_2$ be the stacks on $P_{q_1}$ and $P_{q_2}$ obtained from applying this theorem, and define $(Q_i,\Lambda_i)$ from $U_i$ as in the statement of the claim. The theorem gives (for $i\in\{1,2\}$) $U_i$ is countable, $U_i\in V[H]$, and, without loss of generality, $Q_1 \trianglelefteq Q_2$ and $\Lambda_1 \upharpoonright V[H] = (\Lambda_2)_{Q_1} \upharpoonright V[H]$.\footnote{I.e. $\Lambda_1$ is the restriction of $\Lambda_2$ to trees on $Q_1$.} By Claim \ref{getting sigma from lambda}, $\Theta(\Lambda_1) = \Theta = \Theta(\Lambda_2)$, and therefore $Q_1 = Q_2$ and $\Lambda_1 \upharpoonright V[H] = \Lambda_2 \upharpoonright V[H]$. It remains to show $\Lambda_1 = \Lambda_2$. Since $\Lambda_i$ is a tail strategy of $\Sigma_{q_i}$, there is a tree $T_i\in V[H]$ such that $\Lambda_i = (\rho[T_i])^{M}$. We know $(\rho[T_1])^{V[H]} = (\rho[T_2])^{V[H]}$. Then standard absoluteness arguments imply $(\rho[T_1])^{M} = (\rho[T_2])^{M}$.
    \end{proof}

    By Claim \ref{can coiterate}, we may define a hod pair $(Q,\Lambda)$ to be the direct limit of all countable iterates in $V[H]$ of every $(P_q,\Sigma_q)$ for $q\in Col(\omega,\gamma)$. It follows from Claim \ref{getting sigma from lambda} that $L(Lp^{\Lambda}(\R^*)) = M$. In particular, $Lp^{\Lambda}(\R^*)$ has height $\kappa^+$.

    We have symmetric terms for $Q$ and $\Lambda$, so $Q\in V$ and $\Lambda \upharpoonright V \in V$. Let $\pi$ be a $Col(\omega,<\kappa)$-name (in $V$) for $\R^*$ such that $\pi\subset H_\kappa^V$. Then $Lp^{\Lambda}(\pi\cup H^V_\kappa)\in V$.

    Let $S$ be the result of the $S$-construction\footnote{\cite{sihmor} develops $S$-constructions for ``$\Theta$-g-organized hod mice'' --- the $S$ construction for $Lp^\Sigma$ is defined similarly.} in $Lp^{\Lambda}(\R^*)$ over $\pi\cup H^V_\kappa$. Standard properties of the $S$-construction imply $S \trianglelefteq Lp^{\Lambda}(\pi\cup H^V_\kappa)$. In particular, $Lp^{\Lambda}(\pi\cup H_\kappa^V)$ has height $(\kappa^+)^V$.

    From here the argument is the same as in the proof of Theorem \ref{theta_0 < kappa+}.
\end{proof}

We can generalize Theorem \ref{theta < kappa+} further by applying a new comparison theorem for least branch hod pairs (lbr hod pairs)\footnote{See \cite{lbrbook} for a definition of least branch hod pairs.} due to Sargsyan and Steel:

\begin{theorem}[Sargsyan, Steel]
\label{lbr comparison}
    Assume $AD^+$. Let $(P,\Sigma)$ and $(Q,\Lambda)$ be lbr hod pairs with scope $HC$. Let $x_1$ and $x_2$ be reals coding $P$ and $Q$, respectively. Let $\rho[S_1]$ and $\rho[S_2]$ be Suslin representations of $\Sigma$ and $\Lambda$, respectively. Then there are trees $\T$ and $\U$ on $(P,\Sigma)$ and $(Q,\Lambda)$, respectively, with common last pair $(R,\Omega)$ such that at least one of $\T$ or $\U$ does not drop on its main branch and $\T$, $\U$, and $R$ are countable in $L[x_1,x_2,S_1,S_2]$.
\end{theorem}

See \cite{lbr_comparison} for a proof of Theorem \ref{lbr comparison}. We also need a result of Steel that an lbr hod pair is Suslin-co-Suslin:

\begin{theorem}[Steel]
\label{lbr strategy suslin-co-suslin}
    Assume $AD^+$. If $(P,\Sigma)$ is an lbr hod pair, then $\Sigma$ is Suslin-co-Suslin.
\end{theorem}

Theorem \ref{lbr strategy suslin-co-suslin} is a corollary of Theorem 3.4 of \cite{mouse_pairs_suslin}. Suppose Theorem \ref{lbr strategy suslin-co-suslin} fails. Let $\mathbf{\Gamma}$ be the pointclass of all Suslin sets and $\mathbf{\Delta} = \mathbf{\Gamma} \cap \mathbf{\Gamma}^c$. Letting $\delta = o(\mathbf{\Delta})$, $\delta$ is the largest Suslin cardinal. Let $(P,\Sigma)$ be the minimal lbr hod pair (in the mouse order) such that $\Sigma\notin \mathbf{\Delta}$. Then $(P,\Sigma)$ is a pointclass generator for $\bf{\Delta}$ and we can apply Theorem 3.4. $\mathbf{\Gamma}$ is inductive-like and $\delta$ must be a regular cardinal, so we are in case 3(c) of Theorem 3.4. But this gives a Suslin cardinal larger than $\delta$, a contradiction.

We can now prove a version of Theorem \ref{theta < kappa+} using Theorems \ref{lbr comparison} and \ref{lbr strategy suslin-co-suslin} in place of the analogous results from \cite{hm&msc} for hod pairs below $AD_\R + \Theta$ is regular.

\begin{theorem}
\label{lbr thm}
    Suppose $V\models \neg\square_\kappa$, $\kappa$ is a limit of Woodin cardinals, and $D(V,\kappa)\models ``\Theta_{\alpha+1}$ exists'' $ + $``there is a least branch hod pair $(P,\Sigma)$ such that $P_{\Theta_{\alpha+1}}(\R^*) = Lp^{\Sigma}(\R^*) \cap P(\R^*)$. Then $\Theta^{D(V,\kappa)}_{\alpha+1} < \kappa^+$.
\end{theorem}
\begin{proof}[Proof Sketch]
    Suppose not. Pick an lbr hod pair $(P,\Sigma)$ in $D(V,\kappa)$ which is minimal in the mouse pair order such that $Lp^\Sigma(\R^*) = D(V,\kappa)$ (see Corollary 9.3.7 of \cite{lbrbook}). We pick $(P,\Sigma)$ minimal in this sense so that if $(P',\Sigma')$ is any other lbr hod pair satisfying the same property, then $(P,\Sigma)$ and $(P',\Sigma')$ coiterate to the same lbr hod pair. The rest of the proof is identical to that of Theorem \ref{theta < kappa+}, except for Claim \ref{can coiterate}, which follows from Theorem \ref{lbr comparison} and the remarks above.
\end{proof}

What happens if $\Theta^{D(V,\kappa)}$ is a limit level of the Solovay hierarchy? In this case, Remarks \ref{same suslin-co-suslin sets} and \ref{AD_R equivalences} imply $D(V,\kappa) = olD(V,\kappa)$. So as a corollary to Theorem \ref{result in old derived model}, if $\square_\kappa$ fails, then $\Theta^{D(V,\kappa)} < \kappa^+$. This was already well known by a more elementary argument:

\begin{theorem}
\label{not adr}
    Suppose $\kappa$ is a limit of Woodin cardinals and $D(V,\kappa)\models AD_\R$. Then $\Theta^{D(V,\kappa)}<\kappa^+$.
\end{theorem}
\begin{proof}
    By Remarks \ref{same suslin-co-suslin sets} and \ref{AD_R equivalences}, $P(\R^*) = Hom^*$. The proof of Claim \ref{size of hom} shows that $|Hom^*|^{V[G]} = \kappa$. If $\Theta^{D(V,\kappa)} = \kappa^+$, then the map $A \mapsto w(A)$ is a surjection of $Hom^*$ onto $\kappa^+$ in $V[G]$, contradicting the claim.
\end{proof}

Probably the reader by now also expects the techniques above will go as far up the Solovay hierarchy as progress on studying hod mice allows and shall not be surprised by the following conjecture.

\begin{conjecture}
\label{theta conjecture}
    (PFA) If $\kappa$ is a limit of Woodin cardinals, then $\Theta^{D(V,\kappa)} < \kappa^+$.
\end{conjecture}

\printbibliography

\end{document}